\documentclass[a4paper,12pt]{amsart}
\usepackage{graphicx}
\usepackage{amssymb}
\usepackage{amsmath}
\usepackage{amsthm}
\usepackage{latexsym,autograph2,bezier,amsbsy,enumerate,amsfonts,amsmath,amscd,amssymb,setspace,epic}
\usepackage[latin1]{inputenc}
\newtheorem{defi}{Definition}[section]
\newtheorem{teo}[defi]{Theorem}

\newtheorem{prop}[defi]{Proposition}
\newtheorem{cor}[defi]{Corollary}

\newtheorem{es}[defi]{Example}

\newtheorem{os}[defi]{Remark}

\setloopdiam{6}\setprofcurve{4}

\begin{document}
\title[Generalized crested products of Markov chains]{Generalized crested products of Markov chains}

\author{Daniele D'Angeli}
\address{Daniele D'Angeli. Department of Mathematics, Technion--Israel Institute of
Technology, Technion City, Haifa 32 000.}
\email{dangeli@tx.technion.ac.il}

\author{Alfredo Donno}
\address{Alfredo Donno. Dipartimento di Matematica ``G. Castelnuovo", Sapienza Università di Roma, Piazzale A. Moro, 2, 00185 Roma, Italia.}
\email{donno@mat.uniroma1.it}

\keywords{Reversible Markov chain, generalized crested product,
Insect Markov chain, spectral theory, Gelfand pairs.}
\date{\today}
\maketitle

\begin{abstract}
We define a finite Markov chain, called generalized crested
product, which naturally appears as a generalization of the first
crested product of Markov chains. A complete spectral analysis is
developed and the $k$-step transition probability is given. It is
important to remark that this Markov chain describes a more
general version of the classical Ehrenfest diffusion model.\\
\indent As a particular case, one gets a generalization of the
classical Insect Markov chain defined on the ultrametric space.
Finally, an interpretation in terms of representation group theory
is given, by showing the correspondence between the spectral
decomposition of the generalized crested product and the Gelfand
pairs associated with the generalized wreath product of
permutation groups.
\end{abstract}

\begin{center}
{\footnotesize{\bf Mathematics Subject Classification (2010):}
60J10, 05C50, 06A07, 20B25.}
\end{center}

\section{Introduction}
This paper deals with the study of a finite Markov chain, called
generalized crested product, which is defined on the product of
finite spaces. The generalized crested product is a generalization
of the Markov chains introduced and studied in \cite{crestednoi}
and \cite{noiortogonal}. More precisely, in \cite{crestednoi} the
first crested product of Markov chains is defined, inspired by the
analogous definition for association schemes \cite{crested}, as a
sort of mixing between the classical crossed and nested products
of Markov chains and it contains, as a particular case, the so
called Insect Markov chain introduced by Figà-Talamanca in
\cite{insetto} in the context of Gelfand pairs theory. In
\cite{noiortogonal} a Markov chain on some special structures
called orthogonal block structures is introduced. If the
orthogonal block structure is a poset block structure, then that
Markov chain can also be defined starting from a finite poset
$(I,\leq)$ and it can be interpreted as a slightly different
generalization of the classical Insect Markov chain and of the
associated Gelfand pairs theory. We noticed that, for some
particular posets, the Markov chain in \cite{noiortogonal} has the
same spectral decomposition as the first crested product of Markov
chains despite the corresponding operators do not coincide. So it
is natural to ask if it is possible to define a Markov chain
containing the first crested product as a particular case, giving
rise to an Insect Markov chain for a particular choice of the
parameters involved. This is the reason why this paper aims at
introducing a new Markov chain which can be seen as a modification
of the Markov chain on the orthogonal block structure and the
natural generalization of the first crested product of Markov
chains. The idea is to take a finite poset $(I,\leq)$ and a family
of Markov operators $P_i$ defined on finite sets $X_i$ indexed by
the elements of the poset. Then we consider the sum, over $I$, of
tensor products of Markov chains reflecting, in some sense, the
poset hierarchy structure (see Definition \ref{basicdefinition}).
A necessary and sufficient condition to have reversibility of the
Markov chain is proven in Theorem \ref{teo2}. In Theorem
\ref{decomposizione}, we give a complete spectral analysis of this
Markov chain and we show in Proposition \ref{vannoacoincidere}
that it coincides with the first crested Markov chain when the
poset $(I,\leq)$ satisfies some particular properties. A formula
for the $k$-step probability is given in Section \ref{3.3}.
Moreover, we introduce in Section \ref{newinsect} an Insect Markov
chain on the product $X=\prod_{i\in I} X_i$, naturally identified
with the last level of a graph $\mathcal{T}$ which is the
generalization of the rooted tree. This Insect Markov chain is
obtained from the generalized crested product of Markov chains for
a particular choice of the operators $P_i$, i.e. $P_i=J_i$, where
$J_i$ is the uniform operator on the set $X_i$. If the poset
$(I,\leq)$ is totally ordered, this Insect Markov chain coincides
with the classical Insect Markov chain \cite{insetto}. In Section
\ref{sectiongelfand} we highlight the correspondence with the
Gelfand pairs theory (for a general theory and applications see
\cite{diaconis}): taking the generalized wreath product of
permutation groups \cite{generalized} associated with $(I,\leq)$
and the stabilizer of an element of $X$ under the action of this
group, one gets a Gelfand pair \cite{noiortogonal}, and the
decomposition of the action of the group on the space $L(X)$ into
irreducible submodules is the same as the spectral decomposition
of the Insect Markov chain associated with $(I,\leq)$. This allows
to study many examples of Gelfand pairs only by using basic tools
of linear algebra.\\ \indent It is important to remark that the
generalized crested product can be seen as a generalization of a
classical diffusion model, the Ehrenfest model, as well as of the
$(C,N)$-Ehrenfest model described in \cite{crestednoi} (see
\cite{saloffexc} and \cite{liggett} for more examples and
details).

\section{Preliminaries}

We recall in this section some basic facts about finite Markov
chains (see, for instance, \cite{librotullio}). Let $X$ be a
finite set, with $|X|= m$. Let $P=(p(x,y))_{x,y\in X}$ be a
stochastic matrix, so that
$$
\sum_{x\in X}p(x_0,x) = 1,
$$
for every $x_0 \in X$. Consider the Markov chain on $X$ with
transition matrix $P$. By abuse of notation, we will denote by $P$
this Markov chain as well as the associated Markov operator.

\begin{defi}
The Markov chain $P$ is reversible if there exists a strict
probability measure $\pi$ on $X$ such that
\begin{eqnarray*}
\pi(x) p(x,y) = \pi(y) p(y,x),
\end{eqnarray*}
for all $x,y \in X$.
\end{defi}
If this is the case, we say that $P$ and $\pi$ are in detailed
balance \cite{aldous}.

Define a scalar product on $L(X) = \{f:X \longrightarrow
\mathbb{C}\}$ as
$$
\langle f_1,f_2 \rangle_{\pi} = \sum_{x\in X}f_1(x)
\overline{f_2(x)}\pi(x),
$$
for all $f_1, f_2\in L(X)$, and the linear operator
$P:L(X)\longrightarrow L(X)$ as
\begin{eqnarray}\label{linearoperator}
(Pf)(x) = \sum_{y\in X}p(x,y)f(y).
\end{eqnarray}
It is easy to verify that $\pi$ and $P$ are in detailed balance if
and only if $P$ is self-adjoint with respect to the scalar product
$\langle \cdot, \cdot \rangle_{\pi}$. Under these hypotheses, it
is known that the matrix $P$ can be diagonalized over the reals.
Moreover, $1$ is always an eigenvalue of $P$ and for any
eigenvalue
$\lambda$ one has $|\lambda|\leq 1$.\\
\indent Let $\{\lambda_z\}_{z\in X}$ be the eigenvalues of the
matrix $P$, with $\lambda_{z_0} = 1$. Then there exists an
invertible unitary real matrix $U = (u(x,y))_{x,y \in X}$ such
that $PU = U\Delta$, where $\Delta = (\lambda_x\delta_x(y))_{x,y
\in X}$ is the diagonal matrix whose entries are the eigenvalues
of $P$. This equation gives, for all $x,z \in X$,
\begin{equation}\label{auto1}
\sum_{y \in X}p(x,y)u(y,z) = u(x,z)\lambda_z.
\end{equation}
Moreover, we have $U^T DU = I$, where $D =
(\pi(x)\delta_x(y))_{x,y \in X}$ is the diagonal matrix of
coefficients of $\pi$. This second equation gives, for all $y,z
\in X$,
\begin{equation}\label{auto2}
\sum_{x \in X}u(x,y)u(x,z)\pi(x) = \delta_y(z).
\end{equation}
It follows from \eqref{auto1} that each column of $U$ is an
eigenvector of $P$, and from \eqref{auto2} that these columns are
orthogonal with respect to the product $\langle \cdot,\cdot
\rangle_{\pi}$.

\begin{prop}
The $k$-th step transition probability is given by
\begin{equation}\label{kpassi}
p^{(k)}(x,y) = \pi(y) \sum_{z \in X}u(x,z)\lambda^{k}_zu(y,z),
\end{equation}
for all $x,y \in X$.
\end{prop}

\begin{proof} The proof is a consequence of (\ref{auto1}) and
(\ref{auto2}). In fact, the matrix $U^TD$ is the inverse of $U$,
so that $UU^TD = I$. In formul\ae, we have
$$
\sum_{y \in X}u(x,y)u(z,y) = \frac{1}{\pi(z)}\Delta_z(x).
$$
From the equation $PU = U\Delta$ we get $P = U\Delta U^TD$, which
gives
$$
p(x,y) = \pi(y) \sum_{z \in X}u(x,z)\lambda_zu(y,z).
$$
Iterating this argument we obtain $ P^k = U\Delta^k U^TD$, which
is the assertion.
\end{proof}

Recall that there exists a correspondence between reversible
Markov chains and weighted graphs.

\begin{defi}
A weight on a graph $\mathcal{G}= (X,E)$ is a function $w: X
\times X \longrightarrow [0,+\infty )$ such that
\begin{enumerate}
\item $w(x,y) = w(y,x)$;
\item $w(x,y) > 0$ if and only if $x \sim y$.
\end{enumerate}
\end{defi}

If $\mathcal{G}$ is a weighted graph, it is possible to associate
with $w$ a stochastic matrix $P = (P(x,y))_{x,y \in X}$ on $X$ by
setting
$$
p(x,y) = \frac{w(x,y)}{W(x)},
$$
with $W(x) = \sum_{z \in X}w(x,z)$. The corresponding Markov chain
is called the random walk on $\mathcal{G}$. It is easy to prove
that the matrix $P$ is in detailed balance with the distribution
$\pi$ defined, for every $x \in X$, as
$$
\pi(x) = \frac{W(x)}{W},
$$
with $W = \sum_{z \in X}W(z)$. Moreover, $\pi$ is strictly
positive if $X$ does not contain isolated vertices. The inverse
construction can be done. So, if we have a transition matrix $P$
on $X$ which is in detailed balance with the probability $\pi$,
then we can define a weight $w$ as $w(x,y) = \pi(x)p(x,y)$. This
definition guarantees the symmetry of $w$ and, by setting $E =
\{\{x,y\}: w(x,y) > 0 \}$, we get a weighted graph.\\
\indent There are some important relations between the weighted
graph associated with a transition matrix $P$ and its spectrum
$\sigma(P)$. In fact, it is easy to prove that the multiplicity of
the eigenvalue 1 of $P$ equals the number of connected components
of $\mathcal{G}$. Moreover, the following propositions hold.

\begin{prop}\label{prop5}
Let $\mathcal{G}= (X,E,w)$ be a finite connected weighted graph
and denote $P$ the corresponding transition matrix. Then the
following are equivalent:
\begin{enumerate}
\item $\mathcal{G}$ is bipartite;
\item the spectrum $\sigma (P)$ is symmetric;
\item $-1 \in \sigma (P)$.
\end{enumerate}
\end{prop}

\begin{defi} Let $P$ be a stochastic matrix. $P$ is ergodic if
there exists $n_0 \in \mathbb{N}$ such that
$$
p^{(n_0)}(x,y) > 0, \ \ \ \mbox{for all } x,y \in X.
$$
\end{defi}

\begin{prop}
Let $\mathcal{G} = (X,E)$ be a finite graph. Then the following
conditions are equivalent:
\begin{enumerate}
\item $\mathcal{G}$ is connected and not bipartite;
\item for every weight function on $\mathcal{G}$, the associated transition
matrix $P$ is ergodic.
\end{enumerate}
\end{prop}

So we can conclude that a reversible transition matrix $P$ is
ergodic if and only if the eigenvalue 1 has multiplicity one and
$-1$ is not an eigenvalue. Note that the condition that 1 is an
eigenvalue of $P$ of multiplicity one is equivalent to require
that the probability $P$ is irreducible, according with the
following definition.
\begin{defi}
A stochastic matrix $P$ on a set $X$ is irreducible if, for every
$x_1, x_2 \in X$, there exists $n = n(x_1,x_2)$ such that
$p^{(n)}(x_1,x_2)
>0$.
\end{defi}

\section{Generalized crested product}\label{capitoloprincipale}

\subsection{Definition}\label{3.1}
Let $(I,\leq)$ be a finite poset, with $I=\{1,\ldots, n\}$. The
following definitions are given in \cite{generalized}.
\begin{defi}
A subset $J \subseteq I$ is said
\begin{itemize}
\item {\bf ancestral} if, whenever $i > j$ and $j\in J$, then
$i\in J$; \item {\bf hereditary} if, whenever $i < j$ and $j\in
J$, then $i\in J$; \item {\bf a chain} if, whenever $i, j\in J$,
then either $i\leq j$ or $j\leq i$; \item {\bf an antichain} if,
whenever $i, j\in J$ and $i \neq j$, then neither $i\leq j$ nor
$j\leq i$.
\end{itemize}
\end{defi}

Given an element $i\in I$, we set $A(i)=\{j\in I \ : j>i\}$ to be
the ancestral set of $i$ and $A[i]=A(i)\sqcup\{i\}$. Analogously
we set $H(i)=\{j\in I \ : j<i\}$ to be the hereditary set of $i$
and $H[i]=H(i)\sqcup\{i\}$. For a subset $J\subseteq I$ we put
$A(J)=\bigcup_{j\in J}A(j)$, $A[J]=\bigcup_{j\in J}A[j]$,
$H(J)=\bigcup_{j\in J}H(j)$ and $H[J]=\bigcup_{j\in J}H[j]$.
Moreover, we denote by $\mathbb{S}$ the set of the antichains of
$(I,\leq)$ and we set $\overline{S}=\{j \in I \ : A(j)=\emptyset
\}$. It is clear that $\overline{S}$ and the empty set $\emptyset$
belong to $\mathbb{S}$. Note that $A(\emptyset)=H(\emptyset)=
A[\emptyset] = H[\emptyset] =\emptyset$.

For each $i\in I$, let $X_i$ be a finite set, with $|X_i|=m_i$, so
that we can identify $X_i$ with the set $\{0,1, \ldots, m_i-1\}$.
Moreover, let $P_i$ be an irreducible Markov chain on $X_i$ and
let $p_i$ be the corresponding transition probability. We also
denote by $P_i$ the associated Markov operator
$P_i:L(X_i)\longrightarrow L(X_i)$ defined as in
\eqref{linearoperator}. Let $I_i$ be the identity matrix of size
$m_i$ and set:
$$
J_i = \frac{1}{m_i} \begin{pmatrix}
  1 & 1 & \cdots & 1 \\
  1 & \ddots &  & \vdots \\
  \vdots &  & \ddots & \vdots \\
  1 & \cdots & \cdots & 1
\end{pmatrix}.
$$
We also denote by $I_i$ and $J_i$ the associated Markov operator
on $L(X_i)$, that we call the identity and the uniform operator,
respectively. We are going to define a Markov chain on the set
$X=X_1\times \cdots \times X_n$.

\begin{defi}\label{basicdefinition}
Let $(I,\leq)$ be a finite poset and let $\{p_i^0\}_{i\in I}$ be a
probability distribution on $I$, i.e. $p_i^0>0$ for every $i\in I$
and $\sum_{i=1}^np_i^0=1$. The \textbf{generalized crested
product} of the Markov chains $P_i$ defined by $(I,\leq)$ and
$\{p_i^0\}_{i\in I}$ is the Markov chain on $X$ whose associated
Markov operator is
\begin{eqnarray}\label{definizioneiniziale}
\mathcal{P}=\sum_{i\in I}p_i^0 \left(P_i\otimes
\left(\bigotimes_{j\in H(i)} U_j\right)\otimes
\left(\bigotimes_{j\in I\setminus H[i]} I_j\right)\right).
\end{eqnarray}
\end{defi}

\begin{os}\rm
The generalized crested product can be seen as a generalization of
the classical diffusion Ehrenfest model. This classical model
consists of two urns numbered 0, 1 and $n$ balls numbered $1,
\ldots, n$. A configuration is given by a placement of the balls
into the urns. Note that there is no ordering inside the urns. At
each step, a ball is randomly chosen (with probability $1/n$) and
it is moved to the other urn. In \cite{crestednoi} we generalized
it to the $(C,N)$-Ehrenfest model. Now put $|X_i|=m$, for each
$i=1,\ldots, n$: then we have the following interpretation of the
generalized crested product. Suppose that we have $n$ balls
numbered by $1, \ldots, n$ and $m$ urns. Let $(I,\leq)$ be a
finite poset with $n$ elements. At each step, we choose a ball $i$
according with a probability distribution $p_i^0$: then we move it
to another urn following a transition probability $P_i$ and all
the other balls numbered by indices $j$ such that $j\leq i$ in the
poset $(I,\leq)$ are moved uniformly to a new urn. The balls
corresponding to all the other indices are not moved.
\end{os}
From now on, we suppose that each $P_i$ is in detailed balance
with the probability measure $\sigma_i$.
\begin{teo}\label{teo2}
The generalized crested product is reversible if and only if $P_k$
is symmetric for every $k \in I \setminus \overline{S}$, i.e.
$p_k(x_k,y_k) = p_k(y_k,x_k)$, for all $x_k,y_k\in X_k$. If this
is the case, $\mathcal{P}$ is in detailed balance with the strict
probability measure $\pi$ on $X$ given by
$$
\pi(x_1,\ldots,x_n) = \frac{\prod_{i\in
\overline{S}}\sigma_i(x_i)}{\prod_{i\in
I\setminus\overline{S}}m_i}.
$$
\end{teo}
\begin{proof}
We start by proving that the condition $\sigma_k=\frac{1}{m_k}$,
for each $k\in I\setminus \overline{S}$, is sufficient. Consider
two elements $x = (x_1, \ldots, x_n)$ and $y = (y_1, \ldots, y_n)$
in $X$. First, suppose that there exists $j\in \overline{S}$ such
that $x_j\neq y_j$ and $x_i=y_i$ for all $i\in
\overline{S}\setminus \{j\}$. In this case, we have
\begin{eqnarray*}
\pi(x)\mathcal{P}(x,y)&=&\frac{\prod_{i\in
\overline{S}}\sigma_i(x_i)}{\prod_{i\in I\setminus
\overline{S}}m_i}\cdot
p_j^0\cdot\frac{p_j(x_j,y_j)}{\prod_{i\in H(j)}m_i}\cdot\prod_{i\in I\setminus H[j]}\delta_i(x_i,y_i)\\
&=&\sigma_j(x_j)p_j(x_j,y_j)\cdot p_j^0\cdot\frac{\prod_{i\in
\overline{S}\setminus \{j\}}\sigma_i(x_i)}{\prod_{i\in I\setminus
\overline{S}}m_i}\cdot \frac{\prod_{i\in I\setminus H[j]}\delta_i(x_i,y_i)}{\prod_{i\in H(j)}m_i}\\
&=&\sigma_j(y_j)p_j(y_j,x_j)\cdot p_j^0\cdot\frac{\prod_{i\in
\overline{S}\setminus \{j\}}\sigma_i(y_i)}{\prod_{i\in I\setminus
\overline{S}}m_i}\cdot \frac{\prod_{i\in I\setminus H[j]}\delta_i(y_i,x_i)}{\prod_{i\in H(j)}m_i}\\
&=&\frac{\prod_{i\in \overline{S}}\sigma_i(y_i)}{\prod_{i\in
I\setminus \overline{S}}m_i}\cdot p_j^0\cdot
\frac{p_j(y_j,x_j)}{\prod_{i\in
H(j)}m_i}\cdot\prod_{i\in I\setminus H[j]}\delta_i(y_i,x_i)\\
&=&\pi(y)\mathcal{P}(y,x).
\end{eqnarray*}
If we suppose that there exist $j_1,j_2\in \overline{S}$ such that
$x_{j_h}\neq y_{j_h}$, for $h=1,2$, then
$\mathcal{P}(x,y)=\mathcal{P}(y,x)=0$ and there is nothing to
prove.

Suppose now that $x_i=y_i$ for every $i\in \overline{S}$ and there
is $j\in I\setminus \overline{S}$ such that $x_j\neq y_j$. We have
\begin{eqnarray*}
\pi(x)\mathcal{P}(x,y)&=&\frac{\prod_{i\in
\overline{S}}\sigma_i(x_i)}{\prod_{i\in I\setminus
\overline{S}}m_i}\cdot \left(\sum_{i\in A(j)}
p_i^0\cdot\frac{p_i(x_i,y_i)\prod_{h\in I\setminus H[i]} \delta_h(x_h,y_h)}{\prod_{h\in H(i)}m_h} \right.\\
&+& \left.p_j^0\cdot\frac{p_j(x_j,y_j)\prod_{h\in I\setminus H[j]}\delta_h(x_h,y_h)}{\prod_{h\in H(j)}m_h}\right)\\
&=&\frac{\prod_{i\in \overline{S}}\sigma_i(y_i)}{\prod_{i\in
I\setminus \overline{S}}m_i}\cdot \left(\sum_{i\in A(j)}
p_i^0\cdot\frac{p_i(y_i,x_i)\prod_{h\in I\setminus H[i]}
\delta_h(y_h,x_h)}{\prod_{h\in H(i)}m_h}\right.\\
&+&\left. p_j^0\cdot\frac{p_j(y_j,x_j)\prod_{h\in I\setminus H[j]}\delta_h(y_h,x_h)}{\prod_{h\in H(j)}m_h}\right)\\
 &=&\pi(y)\mathcal{P}(y,x).
\end{eqnarray*}
On the other hand, we show that the condition
$\sigma_k=\frac{1}{m_k}$, for each $k\in I\setminus \overline{S}$,
is necessary.  Suppose that the equality $\pi(x)\mathcal{P}(x,y) =
\pi(y)\mathcal{P}(y,x)$ holds for all $x,y\in X$. Let $i\in
\overline{S}$: by irreducibility, we can choose $x,y\in X$ such
that $x_i\neq y_i$ and $p_i(x_i,y_i)\neq 0$. Let $x_j=y_j$ for
every $j\in \overline{S}\setminus \{i\}$. We have
$$
\pi(x)\mathcal{P}(x,y) = \pi(y)\mathcal{P}(y,x)
\Longleftrightarrow \pi(x)p_{i}(x_{i},y_{i}) =
\pi(y)p_{i}(y_{i},x_{i}).
$$
This gives
\begin{eqnarray}\label{rapporti1}
\frac{\pi(x)}{\pi(y)} = \frac{p_{i}(y_{i}, x_{i})}{p_{i}(x_{i},
y_{i})} = \frac{\sigma_{i}(x_{i})}{\sigma_{i}(y_{i})}.
\end{eqnarray}
Let $\overline{x}\in X$ such that $\overline{x}_j=y_j$ for each
$j\in H(i)$ and $\overline{x}_j=x_j$ for each $j\in I\setminus
H(i)$. Proceeding as above we get
\begin{eqnarray}\label{rapporti2}
\frac{\pi(\overline{x})}{\pi(y)} = \frac{p_{i}(y_{i},
x_{i})}{p_{i}(x_{i}, y_{i})} =
\frac{\sigma_{i}(x_{i})}{\sigma_{i}(y_{i})},
\end{eqnarray}
so that \eqref{rapporti1} and \eqref{rapporti2} imply
$\pi(x)=\pi(\overline{x})$, i.e. $\pi$ does not depend on the
coordinates corresponding to indices in $H(i)$. Let $j\in H(i)$
and let $x'\in X$ such that $ x'_h\neq x_h$ for each $h\in H[j]$
and $x'_k=x_k$ for each $k\in I\setminus H[j]$. The condition
$\pi(x)\mathcal{P}(x,x') = \pi(x')\mathcal{P}(x',x)$ reduces to
\begin{eqnarray}\label{reduced}
\mathcal{P}(x,x') = \mathcal{P}(x',x),
\end{eqnarray}
since $x$ and $x'$ differ only for indices in $H(i)$ and so
$\pi(x)=\pi(x')$. Observe that $x_{\ell} = x_{\ell}'$ for each
$\ell\in I\setminus H[j]$ and so the summands corresponding to
these indices are equal in both sides of \eqref{reduced}.
Moreover, for each $k\in H(j)$, one has $j\in I\setminus H[k]$ and
so the summands corresponding to these indices are 0 in both sides
of \eqref{reduced}, since $x_j'\neq x_j$. Hence, \eqref{reduced}
reduces to $p_j(x_j,x_j')=p_j(x_j',x_j)$, what implies
$\sigma_j(x_j) = \sigma_j(x_j')$ and so the hypothesis of
irreducibility guarantees that $\sigma_j$ is uniform on $X_j$.
This completes the proof.
\end{proof}

\subsection{Spectral analysis}\label{3.2}
The next step is to study the spectral decomposition of the
operator $\mathcal{P}$. Suppose that
$L(X_i)=\bigoplus_{j_i=0}^{r_i}V^{i}_{j_i}$ is the decomposition
of $L(X_i)$ into eigenspaces of $P_i$ and that $\lambda_{j_i}$ is
the eigenvalue corresponding to $V^i_{j_i}$. The eigenspace
$V_0^i$ is the space of the constant functions over
$X_i$: under our hypothesis of irreducibility, we have $\dim(V_0^i)=1$.\\
\indent For every antichain $S = \{i_1,\ldots, i_k \}\in
\mathbb{S}$, define
$$
\mathcal{J}_S:=\{\underline{j}=(j_{i_1},\ldots, j_{i_k})\ : \
j_{i_h}\in\{1,\ldots, r_{i_h}\}\}
$$
and, for $S\in \mathbb{S}$ and $\underline{j}\in \mathcal{J}_S$,
we put
$$
V_{S,\underline{j}}:=V^{i_1}_{j_{i_1}}\otimes\cdots\otimes
V^{i_k}_{j_{i_k}}.
$$
Moreover, we set
\begin{eqnarray}\label{autospazihaifa}
W_{S,\underline{j}}:=V_{S,\underline{j}}\otimes
\left(\bigotimes_{i\in A(S)}
L(X_i)\right)\otimes\left(\bigotimes_{i\in I\setminus A[S]} V^i_0
\right)
\end{eqnarray}
and
\begin{eqnarray}\label{formula7}
\lambda_{S,\underline{j}}=\sum_{h=1}^k
p_{i_h}^0\lambda_{j_{i_h}}+\sum_{i\in I\setminus A[S]}p_i^0.
\end{eqnarray}

\begin{teo}\label{decomposizione}
Let $(I,\leq)$ be a finite poset and let $\mathcal{P}$ be the
generalized crested product defined in
\eqref{definizioneiniziale}. The decomposition of $L(X)$ into
eigenspaces for $\mathcal{P}$ is
$$
L(X)=\bigoplus_{S\in \mathbb{S}}\left( \bigoplus_{\underline{j}\in
\mathcal{J}_S }  W_{S,\underline{j}}\right).
$$
Moreover, the eigenvalue associated with $W_{S,\underline{j}}$ is
$\lambda_{S,\underline{j}}$.
\end{teo}
\begin{proof}
Let $\varphi$ be a function in $W_{S,\underline{j}}$. We can
represent $\varphi$ as the tensor product
$\varphi_1\otimes\cdots\otimes \varphi_n$, with $\varphi_i\in
V^i_{j_i}$ if $i\in S$, $\varphi_i\in L(X_i)$ if $i\in A(S)$, and
$\varphi_i\in V^i_0$ if $i\in I\setminus A[S]$. We have to show
that $(\mathcal{P}\varphi)(x)=\lambda_{S,\underline{j}}\varphi(x)$
for every $x=(x_1,\ldots, x_n)\in X$. One has:
\begin{eqnarray*}
(\mathcal{P}\varphi)(x)&=& \sum_{(y_1,\ldots, y_n)\in X}\sum_{i\in
I} p_i^0\left(p_i(x_i,y_i)\varphi_i(y_i)\frac{\prod_{j\in
I\setminus
H[i]}\delta_j(x_j,y_j)}{\prod_{j\in H(i)}m_j}\prod_{j\neq i}\varphi_j(y_j) \right)\\
&=&\sum_{i\in I}p_i^0\sum_{y_j:\ j\in
H[i]}\frac{p_i(x_i,y_i)\varphi_i(y_i)}{\prod_{j\in
H(i)}m_j}\prod_{j\in I\setminus H[i]}\varphi_j(x_j)\prod_{j\in H(i)}\varphi_j(y_j).\\
\end{eqnarray*}
Observe that, if $i\in S$, then $H(i)\subseteq I\setminus A[S]$
and so $\varphi_j$ is constant for every $j\in H(i)$. Suppose
$i=i_h$, for some $h\in\{1,\ldots, k\}$. Hence, the term of
$\mathcal{P}$ corresponding to $i_h$ is
\begin{eqnarray*}
p_{i_h}^0\prod_{j\in
H(i_h)}m_j\cdot\sum_{y_{i_h}}\frac{p_{i_h}(x_{i_h},y_{i_h})\varphi_{i_h}(y_{i_h})}{\prod_{j\in
H(i_h)}m_j}\prod_{j\neq
i_h}\varphi_j(x_j)=(p_{i_h}^0\lambda_{j_{i_h}})\varphi(x).
\end{eqnarray*}
On the other hand, if $i\in I\setminus A[S]$, then $S\subseteq
I\setminus H[i]$ and so the identity operator $I_j$, for $j\in S$,
acts on the space orthogonal to $V_0^j$ and $P_i$ acts on $V^i_0$.
Note that $H(i)\subseteq I\setminus A[S]$, so that the term
corresponding to the index $i$ is
\begin{eqnarray*}
& &p_{i}^0\prod_{j\in
H(i)}m_j\cdot\sum_{y_{i}}\frac{p_{i}(x_{i},y_{i})\varphi_{i}(y_{i})}{\prod_{j\in
H(i)}m_j}\prod_{j\neq i}\varphi_j(x_j)\\
&=& p_{i}^0\prod_{j\in
H(i)}m_j\cdot\frac{\varphi_{i}(y_{i})}{\prod_{j\in
H(i)}m_j}\prod_{j\neq i}\varphi_j(x_j)=p_i^0\varphi(x).\\
\end{eqnarray*}
Finally, if $i\in A(S)$, then there exists $k\in S$ such that
$k\in H(i)$. In particular, $\varphi_k$ is orthogonal to $V_0^k$,
i.e. $\sum_{y_k\in X_k}\varphi_k(y_k)=0$ and so the term
corresponding to the index $i$ is
\begin{eqnarray*}
& & p_i^0\sum_{y_j:\ j\in
H[i]}\frac{p_i(x_i,y_i)\varphi_i(y_i)}{\prod_{j\in
H(i)}m_j}\prod_{j\in I\setminus H[i]}\varphi_j(x_j)\prod_{j\in
H(i)}\varphi_j(y_j) \\
&=& p_i^0\left(\sum_{y_j:\ k\neq j\in
H[i]}\frac{p_i(x_i,y_i)\varphi_i(y_i)}{\prod_{k\neq j\in
H(i)}m_j}\prod_{j\in I\setminus H[i]}\varphi_j(x_j)\prod_{k\neq
j\in H(i)}\varphi_j(y_j)\right)\cdot \frac{1}{m_k}\sum_{y_k\in
X_k}\varphi_k(y_k)\\
&=& 0.
\end{eqnarray*}
Hence
$$
(\mathcal{P}\varphi)(x)=\left(\sum_{h=1}^k
p_{i_h}^0\lambda_{j_{i_h}}+\sum_{i\in I\setminus
A[S]}p_i^0\right)\varphi(x).
$$
and the claim is proven.
\end{proof}

\begin{cor}
If $P_i$ is ergodic for each $i\in I$, then $\mathcal{P}$ is
ergodic.
\end{cor}
\begin{proof}
The expression of the eigenvalues of $\mathcal{P}$ given in
\eqref{formula7} ensures that the eigenvalue 1 is obtained with
multiplicity one and the eigenvalue $-1$ can never be obtained.
\end{proof}

We are able now to provide the matrices $U, D$ and $\Delta$
associated with $\mathcal{P}$. For every $i$, let $U_i$, $D_i$ and
$\Delta_i$ be the matrices of eigenvectors, of the coefficients of
$\sigma_i$ and of eigenvalues for the probability $P_i$,
respectively. Recall the identification of $X_i$ with the set
$\{0,1,\ldots,m_i-1\}$.

\begin{prop}\label{matriciautovettori} The matrices $U,D$ and $\Delta$ have the following
form:
\begin{itemize}
  \item $U=\sum_{S\in \mathbb{S}}\left(\bigotimes_{i\in S}(U_i-A_i)\right)\otimes\left(\bigotimes_{i\in
  I\setminus A[S]}A_i\right)\otimes\left(\bigotimes_{i\in
  A(S)}I_i^{\sigma_i-norm}\right)$, where
$$I_i^{\sigma_i-norm} = \begin{pmatrix}
  \frac{1}{\sqrt{\sigma_i(0)}} &  &  &  \\
   & \frac{1}{\sqrt{\sigma_i(1)}} &  &  \\
   &  & \ddots &  \\
   &  &  & \frac{1}{\sqrt{\sigma_i(m_i-1)}}
\end{pmatrix}.
$$
By $A_i$ we denote the matrix of size $m_i$ whose entries on the
first column are all 1 and the remaining ones are 0.
\item $D=\bigotimes_{i\in I}D_i$.
\item $\Delta= \sum_{i\in I}p_i^0\Delta_i\otimes \left(\bigotimes_{j\in I\setminus H[i]}
  I_j\right)\otimes \left(\bigotimes_{j\in  H(i)}
  J_j^{diag}\right)$, where $J_j^{diag}$ is the diagonal matrix $\begin{pmatrix}
    1 &  &  &  \\
     & 0 &  &  \\
     &  & \ddots &  \\
     &  &  & 0 \
  \end{pmatrix}$ of size $m_j$.
\end{itemize}
\end{prop}

\begin{proof}
Let us start by proving the statement for the matrix $U$. By
construction, each column of $U$ is an eigenvector of
$\mathcal{P}$. Let us show that the rank of $U$ is maximal. Fix
$S\in \mathbb{S}$. Then the matrix
\begin{eqnarray}\label{matricedec}
\mathcal{M}^S:=\left(\bigotimes_{i\in
S}(U_i-A_i)\right)\otimes\left(\bigotimes_{i\in
  I\setminus A[S]}A_i\right)\otimes\left(\bigotimes_{i\in A(S)}I_i^{\sigma_i-norm}\right)
\end{eqnarray}
has rank $\prod_{j\in S}(m_j-1)\prod_{j\in A(S)}m_j$ if $S\neq
\emptyset$. If $S=\emptyset$, then $\mathcal{M}^S = \otimes_{i\in
I} A_i$ has rank 1. Moreover, eigenvectors arising from different
$S$ are independent because they belong to subspaces of $L(X)$
which are orthogonal with respect to the scalar product $\langle
\cdot, \cdot \rangle_{\pi}$.\\
\indent First, let us show that, if $S\neq S'$, then the sets of
indices corresponding to the non zero columns of $\mathcal{M}^S$
and $\mathcal{M}^{S'}$ are disjoint. Note that $S\neq S'$ implies
$I\setminus A[S]\neq I\setminus A[S']$. Hence, we can assume
without loss of generality that there exists $h\in I\setminus
A[S]$ such that either $h\in S'$ or $h\in A(S')$. Suppose $h\in
S'$, and put $\mathcal{M}^S_k: =\bigotimes_{j\geq k}M^S_j$, where
$M^S_j=U_j-A_j,A_j$ or $I_j^{\sigma_j-norm}$ according with
(\ref{matricedec}). Under our assumption $M^S_h=A_h$ and
$M^{S'}_h=U_h-A_h$, so that our claim is true for
$\mathcal{M}^S_h$ and $\mathcal{M}^{S'}_h$. Then the same property
can be deduced for $\mathcal{M}^S$ and $\mathcal{M}^{S'}$. Suppose
now that $h\in A(S')$. Then there is $h'\in S'$ such that $h\in
A(h')$. We claim that $h'\in I\setminus A[S]$. In fact if $h'\in
S$ then $h\in A(S)$, which is absurd. If $h'\in A(S)$ then $h\in
A(S)$, a contradiction again. Hence, there exists an index $h'\in
S'$ such that $h'\in I\setminus A[S]$ and from \eqref{matricedec}
we deduce that the claim is true for
$\mathcal{M}^S$ and $\mathcal{M}^{S'}$.\\
\indent Hence, we deduce from Theorem \ref{decomposizione} that
the rank of $U$ is $1+\sum_{S\neq \emptyset} \prod_{j\in
S}(m_j-1)\prod_{j\in A(S)}m_j=\prod_{j\in I}m_j$ and so it is
maximal.

In order to get the diagonal matrix $D$, whose entries are the
coefficients of $\pi$, it suffices to consider the tensor product
of the corresponding matrices associated with the probability
$P_i$, for every $i = 1,\ldots, n$.\\ \indent Finally, to get the
matrix $\Delta$ of eigenvalues of $\mathcal{P}$ it suffices to
replace, in the expression of the matrix $\mathcal{P}$, the matrix
$P_i$ by $\Delta_i$ and the matrix $J_i$ by the corresponding
diagonal matrix $J^{diag}_i$.
\end{proof}

\subsection{The case of the first crested product}

In \cite{crestednoi} the definition of the first crested product
of Markov chains is given. More precisely, considering the product
$X_1 \times \cdots \times X_n$ and a partition
\begin{eqnarray}\label{partition}
\{1, \ldots, n\} = C \coprod N
\end{eqnarray}
of the set $\{1, \ldots, n\}$, given a probability distribution
$\{p_i^0\}_{i=1}^n$ on $\{1, \ldots, n\}$, the first crested
product of the Markov chains $P_i$ with respect to the partition
\eqref{partition} is defined as the Markov chain on $X_1 \times
\cdots \times X_n$ whose transition matrix is
\begin{eqnarray*}
P &=& \sum_{i \in C}p_i^0 \left(I_1 \otimes \cdots \otimes
I_{i-1}\otimes P_i \otimes I_{i+1}\otimes \cdots \otimes I_n
\right) \\ &+& \sum_{i\in N}p_i^0 \left( I_1 \otimes \cdots
\otimes I_{i-1}\otimes P_i \otimes J_{i+1} \otimes \cdots \otimes
J_n \right).
\end{eqnarray*}
We want to show in this section that, if the poset $(I,\leq)$
satisfies some special conditions, then the generalized crested
product defined in \eqref{definizioneiniziale} reduces to the
first crested product. We denote by $\preceq$ the usual ordering
of natural numbers.

\begin{prop}\label{vannoacoincidere}
Suppose that $(I,\leq)$ satisfies the following property: given
$i$ such that $H(i)\neq \emptyset$, then $j\in H(i)$ if and only
if $i\prec j$. Then the first crested product of Markov chains is
obtained by the operator defined in (\ref{definizioneiniziale}) by
putting:
$$
N = \{i \ : \ H(i)\neq \emptyset\} \qquad \mbox{and } \qquad C =
\{1,\ldots, n\}\setminus N.
$$
\end{prop}

\begin{proof} The partition $\{1,\ldots,n\} = C\sqcup
N$, with
$$
N = \{i \ : \ H(i)\neq \emptyset\} \qquad \mbox{and } \qquad C =
\{1,\ldots, n\}\setminus N,
$$
gives:
\begin{eqnarray*}
\mathcal{P} &=& \sum_{I\in C}p_i^0I_1\otimes\cdots\otimes I_{i-1}
\otimes P_i\otimes I_{i+1}\cdots \otimes I_n\\ &+& \sum_{I\in
N}p_i^0I_1\otimes\cdots\otimes I_{i-1} \otimes P_i\otimes
J_{i+1}\cdots \otimes J_n
\end{eqnarray*}
and this operator coincides with the first crested product
associated with the partition $\{1,\ldots,n\}=C\sqcup N$. Using
the same notations as in \cite{crestednoi}, we can denote by $i_1$
the minimal element in $N$ with respect to the ordering $\preceq$
of $\{1,\ldots, n\}$. The antichains of the poset $(I,\leq)$ in
this case are:
\begin{itemize}
\item the empty set $\emptyset$;
\item the set $\{i\}$, for every $i \in \{1, \ldots, n\}$;
\item the subsets of $C$;
\item the sets $D\coprod \{i\}$, with $D\subseteq C$ and $i\in
N$ such that $d \prec i$ for every $d\in D$.
\end{itemize}
The eigenspaces associated with antichains which are subsets of
the set $\{1,2,\ldots, i_1\}$ are exactly the eigenspaces of
second type described in Theorem 4.3 of \cite{crestednoi}. All the
other antichains yield eigenspaces of first type.
\end{proof}

\begin{es}\rm Consider the following diagram:
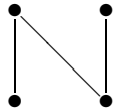
\begin{center}
\begin{picture}(400,30)
\letvertex A=(185,30)\letvertex B=(185,0)
\letvertex C=(215,0)\letvertex D=(215,30)
\drawvertex(A){$\bullet$}\drawvertex(B){$\bullet$}
\drawvertex(C){$\bullet$}\drawvertex(D){$\bullet$}
\drawundirectededge(A,B){}\drawundirectededge(A,C){}\drawundirectededge(C,D){}
\end{picture}
\end{center}
and put on its vertices the labelling
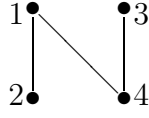
\begin{center}
\begin{picture}(400,30)
\letvertex A=(185,30)\letvertex B=(185,0)
\letvertex C=(215,0)\letvertex D=(215,30)
\drawvertex(A){$\bullet$}\drawvertex(B){$\bullet$}
\drawvertex(C){$\bullet$}\drawvertex(D){$\bullet$}
\put(177,25){1}\put(177,-2){2}\put(218,25){3}\put(218,-2){4}
\drawundirectededge(A,B){}\drawundirectededge(A,C){}\drawundirectededge(C,D){}
\end{picture}
\end{center}
This defines a poset $(I,\leq)$, with $I=\{1,2,3,4\}$,
$H(1)=\{2,4\}$, $H(2)=H(4)=\emptyset$, $H(3)=\{4\}$, whose
associated generalized crested product is (see
\eqref{definizioneiniziale})
$$
\mathcal{P} = p_1^0P_1\otimes U_2 \otimes I_3 \otimes U_4 +p_2^0
I_1\otimes P_2\otimes I_3\otimes I_4+p_3^0I_1\otimes I_2 \otimes
P_3\otimes U_4 +p_4^0I_1\otimes I_2 \otimes I_3\otimes P_4.
$$
On the other hand, there exists no partition $\{1,2,3,4\} =
C\sqcup N$ such that the associated first crested product
coincides with $\mathcal{P}$. In fact, it is not difficult to
check that no labelling satisfying the conditions of Proposition
\ref{vannoacoincidere} can be given to the vertices of this
diagram.
\end{es}

\subsection{$k$-step transition probability}\label{3.3} We provide here
an explicit formula for the $k$-step transition probability. Let
$x=(x_1,\ldots,x_n)$ and $y=(y_1,\ldots,y_n)$ be two elements of
$X$. From \eqref{kpassi} and Proposition \ref{matriciautovettori},
we get
\begin{eqnarray*}
p^{(k)}(x,y)&=&\pi(y)\cdot\sum_{z\in X}\left[
\left(\sum_{S}\prod_{i\in S}(u_i-a_i)(x_i,z_i) \prod_{i\in
A(S)}\delta_{\sigma_i}(x_i,z_i)
\prod_{i\in I\setminus A[S]}a_i(x_i,z_i)\right)\right.\\
 &\times &\lambda_z^k  \left. \left(\sum_{S}\prod_{i\in
S}(u_i-a_i)(y_i,z_i) \prod_{i\in A(S)}\delta_{\sigma_i}(y_i,z_i)
\prod_{i\in I\setminus A[S]}a_i(y_i,z_i)\right)\right].\\
\end{eqnarray*}
This formula becomes simpler by putting $x=0=(0,\ldots,0)$. We get
\begin{eqnarray*}
p^{(k)}(0,y)&=&\pi(y)\cdot\sum_{S}\sum_{\begin{array}{c}\scriptstyle z \textrm{ s.t. }z_i\neq 0, \ i\in S \\
\scriptstyle z_h=0, \ h\not\in S
\end{array}}\left(\frac{\prod_{i\in S}(u_i-a_i)(0,z_i)}{
\prod_{j\in A(S)}\sqrt{\sigma_j(0)}}
\right)\lambda_z^k\\
 &  \times & \left(\sum_{S}\prod_{i\in
S}(u_i-a_i)(y_i,z_i) \prod_{i\in
A(S)}\delta_{\sigma_i}(y_i,z_i)\prod_{i\in I\setminus
A[S]}a_i(y_i,z_i)
\right).\\
\end{eqnarray*}
Note that this expression consists of no more than $ 1+\sum_{S\neq
\emptyset} \prod_{i\in S}(m_i-1)$ terms.

\section{Generalized Insect Markov chain}\label{newinsect}
In this section we describe a Markov chain which is a
generalization of the so called Insect Markov chain introduced in
\cite{insetto} and obtained in \cite{crestednoi} as a particular
case of first crested product. In \cite{noiortogonal}, we extended
it to more general structures called orthogonal block structures.
In particular, we observed that if the orthogonal block structure
is a poset block structure, then the Insect Markov chain can be
defined starting from a finite poset $(I,\leq)$. Our aim is to
define a new Insect Markov chain on some structures starting from
a finite poset $(I,\leq)$, and to check that it can be obtained
from the operator $\mathcal{P}$ defined in Section
\ref{capitoloprincipale} for a particular choice of the
probability distribution $\{p_i^0\}_{i\in I}$ and of the operators $P_i$.\\
\indent Given a finite poset $(I,\leq)$, one can naturally
associate with each antichain $S\in \mathbb{S}$ an ancestral set
$A_S$ by setting $A_S:=I\setminus H[S]$ and it is not difficult to
show that this correspondence is bijective (see
\cite{noiortogonal}). Moreover, we can give to the set of
ancestral subsets of $(I,\leq)$ a natural structure of a poset by
putting $A_S\leq A_{S'}$ if $A_S\supseteq A_{S'}$. In particular,
it is clear that for each $i$ in $I$ the set $\{i\}$ is an
antichain. If we only take antichains constituted by singletons,
the poset of ancestral subsets that we obtain is naturally
isomorphic to $(I,\leq)$. In fact, it is easy to check that
$A_i\supseteq A_j$ if and only if $i\leq j$. We add to this poset
the minimal ancestral set $I$ and denote this poset by
$(I_\mathcal{A},\leq)$. We fix now our attention on the maximal
chains contained in $(I_\mathcal{A},\leq)$. By maximal chain we
mean a chain in $(I_{\mathcal{A}},\leq)$ to which no ancestral set
can be added without losing the property of being totally ordered.

Our aim is to construct a new graph $\mathcal{T}$ obtained by
gluing together some trees arising from the construction we are
performing. Let $X_i=\{0,1,\ldots, m_i-1\}$ and put, as usual,
$X=X_1\times \cdots\times X_n$. For each ancestral set $A\in
(I_\mathcal{A},\leq)$, the ancestral relation $\sim_A$ on $X$ (see
\cite{generalized}) is defined as
$$
x\sim_A y \qquad \mbox{if} \qquad x_i=y_i \qquad \forall \ i\in A.
$$
Observe that the cardinality of an equivalence class of the
relation $\sim_A$ is $\prod_{i\not\in A }m_i$; the cardinality is
1 if $A=I$ (in other words, $\sim_I$ corresponds to the equality
relation). Consider a maximal chain $\{I,A_1,\ldots, A_k\}$ in
$(I_{\mathcal{A}},\leq)$. Start with the set $X$ corresponding to
the relation $\sim_I$ and create a new level corresponding to the
ancestral set $A_1$, in such a way that all elements of $X$ that
are $\sim_{A_1}$-equivalent have a common father in this new
level; iterate this construction for all the ancestral sets in the
chain.\\ \indent For each maximal chain in
$(I_{\mathcal{A}},\leq)$, the arising structure is a disjoint
union of finitely many trees; the final step is to glue together
the structures arising from different maximal chains, by
identifying vertices corresponding to the same ancestral set in
$(I_{\mathcal{A}},\leq)$. We get a graph that we call
$\mathcal{T}$.

\begin{es}\rm
Consider the following simple poset $(I,\leq)$:
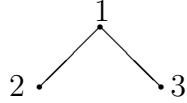
\begin{center}
\begin{picture}(400,28)
\put(170,25){\circle*{2}}
\put(150,5){\circle*{2}}\put(190,5){\circle*{2}}
\put(168,27){$1$}\put(140,2){$2$}\put(193,2){$3$}
\put(170,25){\line(-1,-1){20}}\put(170,25){\line(1,-1){20}}
\end{picture}
\end{center}
Since $|I|=3$, the poset contains three antichains which are
singleton: $S_1=\{1\}$, $S_2=\{2\}$ and $S_3=\{3\}$. The
corresponding ancestral sets are $A_1=\emptyset$, $A_2=\{1,3\}$
and $A_3=\{1,2\}$, so that the associated ancestral poset
$(I_{\mathcal{A}},\leq)$ is
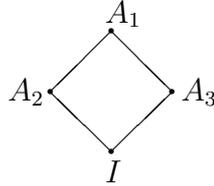
\begin{center}
\begin{picture}(400,53)
\put(170,50){\circle*{2}}
\put(150,30){\circle*{2}}\put(190,30){\circle*{2}}
\put(170,10){\circle*{2}}
\put(168,53){$A_1$}\put(136,27){$A_2$}\put(193,27){$A_3$}
\put(168,0){$I$}
\put(170,50){\line(-1,-1){20}}\put(170,50){\line(1,-1){20}}
\put(170,10){\line(-1,1){20}}\put(170,10){\line(1,1){20}}
\end{picture}
\end{center}
Suppose that $X_1=X_2=X_3=\{0,1\}$, so that $X =
\{000,001,010,011,100,101,110,111\}$. The partitions of $X$
corresponding to the ancestral equivalence relations defined by
$A_1,A_2,A_3$ are:
\begin{itemize}
\item $\sim_{A_1} = X$, since $\sim_{A_1}$ is the universal relation;
\item $\sim_{A_2} = \{000,010\} \coprod \{001,011\} \coprod \{100,110\} \coprod
\{101,111\}$;
\item $\sim_{A_3} = \{000,001\} \coprod \{010,011\} \coprod \{100,101\} \coprod
\{110,111\}$.
\end{itemize}
The trees associated with the maximal chains $\{I,A_3,A_1\}$ and
$\{I,A_2,A_1\}$ are, respectively,
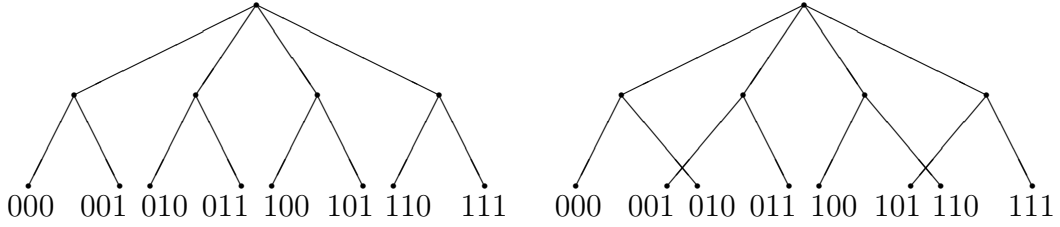
\begin{center}
\begin{picture}(400,75)
\put(105,75){\circle*{2}}
\put(45,45){\circle*{2}}\put(85,45){\circle*{2}}
\put(125,45){\circle*{2}}
\put(165,45){\circle*{2}}\put(30,15){\circle*{2}}\put(60,15){\circle*{2}}\put(70,15){\circle*{2}}
\put(100,15){\circle*{2}}
\put(110,15){\circle*{2}}\put(140,15){\circle*{2}}\put(150,15){\circle*{2}}\put(180,15){\circle*{2}}

\put(23,5){000}\put(47,5){001}\put(67,5){010} \put(87,5){011}
\put(107,5){100}\put(128,5){101}\put(147,5){110}\put(172,5){111}

\put(105,75){\line(-2,-1){60}}\put(105,75){\line(-2,-3){20}}
\put(105,75){\line(2,-3){20}}\put(105,75){\line(2,-1){60}}

\put(45,45){\line(-1,-2){15}}\put(45,45){\line(1,-2){15}}
\put(85,45){\line(-1,-2){15}}\put(85,45){\line(1,-2){15}}\put(125,45){\line(-1,-2){15}}
\put(125,45){\line(1,-2){15}}\put(165,45){\line(-1,-2){15}}\put(165,45){\line(1,-2){15}}

%%%%%%%%%%%%%%%%%%%%%%%%%%%%%%%%%%%%%%%%%%%%%%%%%%%%%%%%%%%%%%%%%%%%%%%%%%%%%
\put(285,75){\circle*{2}}\put(225,45){\circle*{2}}\put(265,45){\circle*{2}}
\put(305,45){\circle*{2}}
\put(345,45){\circle*{2}}\put(210,15){\circle*{2}}\put(240,15){\circle*{2}}\put(250,15){\circle*{2}}
\put(280,15){\circle*{2}}
\put(290,15){\circle*{2}}\put(320,15){\circle*{2}}\put(330,15){\circle*{2}}\put(360,15){\circle*{2}}

\put(203,5){000}\put(227,5){001}\put(247,5){010} \put(267,5){011}
\put(287,5){100}\put(308,5){101}\put(327,5){110}\put(352,5){111}

\put(285,75){\line(-2,-1){60}}\put(285,75){\line(-2,-3){20}}
\put(285,75){\line(2,-3){20}}\put(285,75){\line(2,-1){60}}

\put(225,45){\line(-1,-2){15}}\put(225,45){\line(5,-6){25}}
\put(265,45){\line(-5,-6){25}}\put(265,45){\line(1,-2){15}}
\put(305,45){\line(-1,-2){15}}\put(305,45){\line(5,-6){25}}
\put(345,45){\line(-5,-6){25}}\put(345,45){\line(1,-2){15}}
\end{picture}
\end{center}
Finally, the graph $\mathcal{T}$ obtained by gluing them is
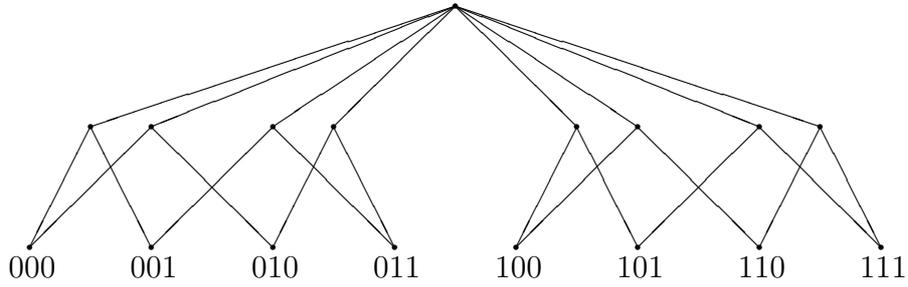
\begin{center}
\begin{picture}(400,125)
\put(195,120){\circle*{2}}\put(75,80){\circle*{2}}
\put(95,80){\circle*{2}}
\put(135,80){\circle*{2}}\put(155,80){\circle*{2}}
\put(235,80){\circle*{2}}\put(255,80){\circle*{2}}
\put(295,80){\circle*{2}} \put(315,80){\circle*{2}}
\put(55,40){\circle*{2}}\put(95,40){\circle*{2}}\put(135,40){\circle*{2}}
\put(175,40){\circle*{2}}\put(255,40){\circle*{2}}\put(295,40){\circle*{2}}
\put(215,40){\circle*{2}}\put(335,40){\circle*{2}}

\put(48,30){000}\put(88,30){001}\put(128,30){010}
\put(168,30){011}\put(248,30){101}\put(288,30){110}
\put(208,30){100}\put(328,30){111}

\put(195,120){\line(-3,-1){120}}
\put(195,120){\line(-5,-2){100}}\put(195,120){\line(-3,-2){60}}
\put(195,120){\line(-1,-1){40}}

\put(195,120){\line(3,-1){120}}\put(195,120){\line(5,-2){100}}
\put(195,120){\line(3,-2){60}}\put(195,120){\line(1,-1){40}}
%%%
\put(75,80){\line(-1,-2){20}}
\put(75,80){\line(1,-2){20}}\put(95,80){\line(-1,-1){40}}
\put(95,80){\line(1,-1){40}}\put(135,80){\line(-1,-1){40}}
\put(135,80){\line(1,-1){40}}\put(155,80){\line(-1,-2){20}}
\put(155,80){\line(1,-2){20}}

\put(235,80){\line(-1,-2){20}}
\put(235,80){\line(1,-2){20}}\put(255,80){\line(-1,-1){40}}
\put(255,80){\line(1,-1){40}}\put(295,80){\line(-1,-1){40}}
\put(295,80){\line(1,-1){40}}\put(315,80){\line(-1,-2){20}}
\put(315,80){\line(1,-2){20}}
\end{picture}
\end{center}
\end{es}
\begin{os}\rm Compare the graph $\mathcal{T}$ with the more
complicated poset in \cite{noiortogonal}, Figure 3, that can be
obtained by considering all the antichains of $(I,\leq)$. However,
Theorem \ref{spettraleinsetonegenerale} implies that the spectral
decompositions coincide in the two constructions.
\end{os}
The \textbf{generalized Insect Markov chain} is the Markov chain
on $X$ obtained by thinking of an insect performing a simple
random walk on $\mathcal{T}$. Starting from an element in $X$
(naturally identified with the bottom level via the identity
relation), the next stopping time is when another vertex in $X$ is
reached by the insect in the simple random walk on $\mathcal{T}$.
In order to describe this Markov chain we introduce some notations
and useful coefficients having a probabilistic meaning.

Observe that moving to an upper level in $\mathcal{T}$ means to
pass in $(I_{\mathcal{A}},\leq)$ from the ancestral set $A_i$ to
an ancestral set $A_j$ such that $A_i\lhd A_j$, where $\lhd$ means
that there is no ancestral set between $A_i$ and $A_j$ (we have
$|\{A_k\in I_{ \mathcal{A}}:A_i\lhd A_k\}|$ possibilities), moving
to a lower level in $\mathcal{T}$ from the ancestral set $A_j$
means to pass to an ancestral set $A_i$ such that $A_i\lhd A_j$
(these are $\sum_{A_i\in I_{\mathcal{A}}:A_i\lhd A_j}\prod_{k\in
A_i\setminus A_j}m_k$ possibilities).

Let $A_i\lhd A_j$ and let $\alpha_{i,j}$ be the probability of
moving from the ancestral $A_i$ to the ancestral $A_j$. The
following relation is satisfied:
\begin{eqnarray}\label{alpha}
\alpha_{i,j} &=&    \frac{1}{\sum_{A_k\in I_{\mathcal{A}}: A_k\lhd
A_i}\prod_{h\in A_k\setminus
A_i}m_h+|\{A_l\in I_{\mathcal{A}}:A_i\lhd A_l\}|} \\
&+& \sum_{A_k\in I_{\mathcal{A}}: A_k\lhd A_i}\frac{\prod_{h\in
A_k\setminus A_i}m_h}{ {\sum_{A_k\in I_{\mathcal{A}}: A_k\lhd
A_i}\prod_{h\in A_k\setminus A_i}m_h+|\{A_l\in
I_{\mathcal{A}}:A_i\lhd A_l\}|}}\alpha_{k,i}\alpha_{i,j}.\nonumber
\end{eqnarray}
In fact, the insect can directly pass from $A_i$ to $A_j$ with
probability $\alpha_{i,j}$ or go down to any $A_k$ such that
$A_k\lhd A_i$, and then come back to $A_i$ with probability
$\alpha_{k,i}$, and one starts the recursive argument. From direct
computations, one gets
\begin{eqnarray*}
\alpha_{I,i} = \frac{1}{|\{A_k\in I_{\mathcal{A}}: I \lhd A_k\}|},
\end{eqnarray*}
Moreover, if $\alpha_{I,i} =1$ we have, for all $A_j$ such that
$A_i\lhd A_j$,
\begin{eqnarray*}
\alpha_{i,j} = \frac{1}{\sum_{A_k\in I_{\mathcal{A}}: A_k\lhd
A_i}\prod_{h\in A_k\setminus A_i}m_h+|\{A_l\in
I_{\mathcal{A}}:A_i\lhd A_l\}|}.
\end{eqnarray*}
If $\alpha_{I,i}\neq 1$, the coefficient $\alpha_{i,j}$ is defined
as in (\ref{alpha}).\\
\\
Set
\begin{eqnarray}\label{coefficienti}
p_i = \sum_{\begin{array}{c}\scriptstyle C\subseteq I_{\mathcal{A}} \ \textrm{chain} \\
\scriptstyle C = \{I,A_j,\ldots,A_k,A_i\}
\end{array}}\alpha_{I,j}\cdots \alpha_{k,i}\left(1-\sum_{A_i\lhd
A_l}\alpha_{i,l}\right)
\end{eqnarray}
and observe that $p_i$ expresses the probability of reaching the
ancestral $A_i$ but not $A_l$ such that $A_i\lhd A_l$ in
$(I_{\mathcal{A}},\leq)$. Moreover, we put
\begin{eqnarray}\label{coefficientiginevra}
p_i = \sum_{\begin{array}{c}\scriptstyle C\subseteq I_{\mathcal{A}} \ \textrm{chain} \\
\scriptstyle C = \{I,A_j,\ldots,A_k,A_i\}
\end{array}}\alpha_{I,j}\cdots \alpha_{k,i}
\end{eqnarray}
if $A_i$ is a maximal element of $(I_{\mathcal{A}},\leq)$.
\begin{prop}
Let $(I,\leq)$ be a finite poset. The generalized Insect Markov
chain coincides with the generalized crested product $\mathcal{P}$
where the coefficients $p_i^0$ are chosen as in
(\ref{coefficienti}) (or (\ref{coefficientiginevra})) and the
operators $P_i$ are the uniform operators $J_i$.
\end{prop}
\begin{proof}
Let $x,y \in X$ and  let $p$ be the transition probability of the
generalized Insect Markov chain. By construction of the poset
$(I_{\mathcal{A}},\leq)$, we have
\begin{eqnarray*}
p(x,y)= \sum_{\begin{array}{c}\scriptstyle  A_i\in I_{\mathcal{A}}, \\
\scriptstyle x\sim_{A_i} y
\end{array}}\sum_{\begin{array}{c}\scriptstyle
C\subseteq I_{\mathcal{A}} \ \textrm{chain} \\
\scriptstyle C = \{I,A_j,\ldots,A_k,A_i\}
\end{array}}\frac{\alpha_{I,j}\cdots \alpha_{k,i}\left(1-\sum_{A_i\lhd
A_l}\alpha_{i,l}\right)}{\prod_{h\not\in A_i}m_h}.
\end{eqnarray*}
The summand corresponding to $A_i$ can be represented as
$$
p_i\cdot\left(\bigotimes_{j\in I\setminus H[i]}
I_j(x_j,y_j)\right)\otimes\left(\bigotimes_{j\in H[i]}
J_j(x_j,y_j)\right),
$$
which is the $i$-th term of the operator $\mathcal{P}$, where
$p_i^0$ is chosen as in (\ref{coefficienti}) (or
(\ref{coefficientiginevra})) and $P_i=J_i$.
\end{proof}

If $P_i=J_i$, then the spectral decomposition of $L(X_i)$ is
$L(X_i)=V^i_0\oplus V^i_1$, with $V^i_1=\{f\in L(X_i) :
\sum_{x_i\in X_i}f(x_i)=0\}$. Theorem \ref{decomposizione} implies
the following result.
\begin{teo}\label{spettraleinsetonegenerale}
Let $(I,\leq)$ be a finite poset, $\mathbb{S}$ be the set
antichains and $p_i$ be as in (\ref{coefficienti}) (or
(\ref{coefficientiginevra})). Let $P$ be the corresponding
generalized Insect Markov chain on $X$. Then the eigenspaces of
$P$ are
$$
W_S=\left(\bigotimes_{i\in
S}V_1^i\right)\otimes\left(\bigotimes_{i\in
A(S)}L(X_i)\right)\otimes\left(\bigotimes_{i\in I\setminus
A[S]}V_0^i\right), \qquad \mbox{for each } S\in \mathbb{S},
$$
with associated eigenvalue
$$
\lambda_S=\sum_{i\in I\setminus A[S]} p_i.
$$
\end{teo}

Recall that two posets $(I,\leq)$ and $(J,\preceq)$ are isomorphic
if there exists an \textit{order-preserving} bijection
$\varphi:I\longrightarrow J$, i.e.
$$
x\leq y \qquad \Longleftrightarrow
\qquad\varphi(x)\preceq\varphi(y) \qquad \mbox{for all }x,y \in I.
$$
Now let $S,S'$ be two antichains of $(I,\leq)$. It is easy to
verify that, if there exists an automorphism $\varphi$ of the
poset $(I,\leq)$ such that $\varphi(S)=S'$, then
\begin{eqnarray}\label{eigenvaluesequal}
\lambda_S=\lambda_{S'}.
\end{eqnarray}
In fact, if this is the case, then $i\in I\setminus A[S]$ if and
only if $\varphi(i)\in I\setminus A[S']$. Moreover, one has $p_{i}
= p_{\varphi(i)}$ for each $i\in I\setminus A[S]$, since $\varphi$
is order-preserving and \eqref{eigenvaluesequal} follows. On the
other hand, the existence of such an automorphism is not a
necessary condition in order to have \eqref{eigenvaluesequal}, as
the following example shows.
\begin{es}\rm Consider the poset $(I,\leq)$ and the associated ancestral poset $(I_{\mathcal{A}},\leq)$ in the
pictures below.
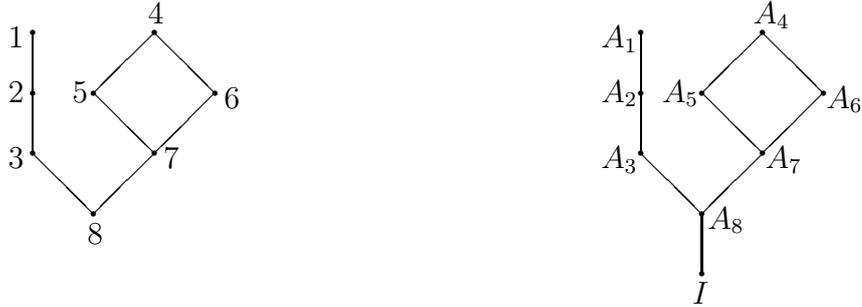
\begin{center}
\begin{picture}(400,95)
\put(80,90){\circle*{2}} \put(120,90){\circle*{2}}
\put(80,70){\circle*{2}} \put(100,70){\circle*{2}}
\put(140,70){\circle*{2}}\put(80,50){\circle*{2}}\put(120,50){\circle*{2}}\put(100,30){\circle*{2}}
\put(280,90){\circle*{2}}
\put(320,90){\circle*{2}}\put(280,70){\circle*{2}}\put(300,70){\circle*{2}}\put(340,70){\circle*{2}}
\put(280,50){\circle*{2}}
\put(320,50){\circle*{2}}\put(300,30){\circle*{2}}\put(300,10){\circle*{2}}

\put(72,85){$1$}\put(72,67){$2$}\put(72,45){$3$} \put(118,93){$4$}
\put(93,67){$5$}\put(143,65){$6$}\put(123,45){$7$}\put(98,20){$8$}
\put(267,85){$A_1$}\put(267,67){$A_2$}\put(267,45){$A_3$}
\put(317,93){$A_4$}
\put(287,67){$A_5$}\put(341,65){$A_6$}\put(321,45){$A_7$}\put(302,25){$A_8$}\put(297,0){$I$}

\put(80,90){\line(0,-1){20}}\put(80,70){\line(0,-1){20}}
\put(80,50){\line(1,-1){20}}\put(120,90){\line(-1,-1){20}}
\put(120,90){\line(1,-1){20}}\put(100,70){\line(1,-1){20}}
\put(140,70){\line(-1,-1){20}}\put(120,50){\line(-1,-1){20}}

\put(280,90){\line(0,-1){20}}\put(280,70){\line(0,-1){20}}
\put(280,50){\line(1,-1){20}}\put(320,90){\line(-1,-1){20}}
\put(320,90){\line(1,-1){20}}\put(300,70){\line(1,-1){20}}
\put(340,70){\line(-1,-1){20}}\put(320,50){\line(-1,-1){20}}
\put(300,30){\line(0,-1){20}}
\end{picture}
\end{center}
Suppose that $|X_i|=m$ for each $i=1,\ldots, 8$. Then it is not
difficult to prove that:
\begin{enumerate}
\item $\lambda_{\{3\}} = \lambda_{\{7\}}$;
\item $\lambda_{\{3,5\}} = \lambda_{\{3,6\}}$;
\item $\lambda_{\{2,5\}} = \lambda_{\{2,6\}}$;
\item $\lambda_{\{1,5\}} = \lambda_{\{1,6\}}$.
\end{enumerate}
Observe that the antichains $\{3\}$ and $\{7\}$ verify
$\lambda_{\{3\}} = \lambda_{\{7\}}$ but there exists no
automorphism of $(I,\leq)$ mapping $\{3\}$ to $\{7\}$.
\end{es}

\section{Gelfand Pairs}\label{sectiongelfand}
Consider the generalized crested product defined in
(\ref{definizioneiniziale}) obtained by choosing $P_i=J_i$ for
each $i\in I$. We know that in this case one has
$$
L(X_i) = V^i_0 \oplus V^i_1 \qquad \mbox{ for each }i\in I,
$$
where $V^i_0\cong \mathbb{C}$ is the space of constant functions
on $X_i$ and $V^i_1=\{f\in L(X_i) : \sum_{x_i\in X_i}f(x_i)=0\}$.
Hence, $\dim (V^i_0) = 1$ and
$\dim (V^i_1) = m_i-1$.\\
\indent In \cite{crestednoi} we made the following remarks: for
the crossed product, the eigenspaces of the operator coincide with
the irreducible submodules of the representation of the direct
product $Sym(m_1)\times \cdots \times Sym(m_n)$ over $L(X_1\times
\cdots \times X_n)$; for the nested product, the eigenspaces of
the operator coincide with the irreducible submodules of the
representation of the wreath product $Sym(m_n)\wr \cdots \wr
Sym(m_1)$ over $L(X_1\times \cdots \times X_n)$.\\ \indent It is
natural to ask if such a correspondence can be extended to the
general case. Actually, the answer is positive and it is given by
a family of groups containing, as particular cases, both the
direct product and the wreath product of permutation groups. These
groups are the so called {\bf generalized wreath product},
introduced in \cite{generalized} as permutation groups of the so
called poset block structures.\\ \indent In \cite{generalized} it
is proven that, given $n$ finite spaces $X_i$, indexed by the
elements of a finite poset $(I,\leq)$, the action of the
generalized wreath product of the permutation groups $Sym(X_i)$ on
the space $L(X_1\times \cdots \times X_n)$ has the following
decomposition into irreducible submodules:
$$
L(X_1\times \cdots \times X_n) = \bigoplus_{S\subseteq I \
\textrm{antichain}}W_S,
$$
with
\begin{eqnarray*}
W_S=\left(\bigotimes_{i\in A(S)}L(X_i)\right)\otimes
\left(\bigotimes_{i\in S}V^i_1\right)\otimes\left(\bigotimes_{i\in
I\setminus A[S]}V^i_0\right).
\end{eqnarray*}
These irreducible submodules coincide with the eigenspaces that we
described in
(\ref{autospazihaifa}) if $P_i=J_i$ and this answers our question.\\
\\
Finally, we proved in \cite{noiortogonal} that the action of the
generalized wreath product of the groups $Sym(X_i)$ on
$L(X_1\times \cdots \times X_n)$ yields symmetric Gelfand pairs
(see \cite{librotullio} or \cite{estate} for the definition) when
one considers the subgroup stabilizing a given element $x_0 =
(x_0^1,\dots, x^n_0)\in X_1\times \cdots \times X_n$. Moreover,
the spherical function associated with $W_S$ is
\begin{eqnarray}\label{spherical}
\phi_S= \bigotimes_{i\in A(S)}\varphi_i \bigotimes_{i\in S}\psi_i
\bigotimes_{i\in I\setminus A[S]}\varrho_i,
\end{eqnarray}
where $\varphi_i,\psi_i\in L(X_i)$ are defined as
$$
\varphi_i(x) =
  \begin{cases}
    1 & x = x^i_0 \\
    0 & \text{otherwise}
  \end{cases},
\qquad \psi_i(x) =
  \begin{cases}
    1 & x = x^i_0 \\
    -\frac{1}{m_i-1} & \text{otherwise}
  \end{cases}
$$
and $\varrho_i\in L(X_i)$ satisfies $\varrho_i(x_i) = 1$ for every
$x_i \in X_i$.


\begin{thebibliography}{99}

\bibitem{aldous} D. Aldous and J. Fill, {\it Reversible Markov Chains and Random Walk on
Graphs}, monograph in preparation
(\texttt{http://www.stat.berkeley.edu/users/aldous/RWG/book.html}).

\bibitem{crested} R. A. Bailey and P. J. Cameron, Crested products of association schemes,
\textit{J. London Math. Soc.} (2) {\bf 72} (2005), no. 1, 1--24.

\bibitem{generalized} R. A. Bailey, Cheryl E. Praeger, C. A. Rowley and T. P. Speed, Generalized
wreath products of permutation groups. {\it Proc. London Math.
Soc.} (3) {\bf 47} (1983), no. 1, 69--82.

\bibitem{librotullio} T. Ceccherini-Silberstein, F. Scarabotti and F.
Tolli, {\it Harmonic Analysis on Finite Groups: Representation
Theory, Gelfand Pairs and Markov Chains}, Cambridge Studies in
Advanced Mathematics {\bf 108}, Cambridge University Press,
Cambridge, 2008.

\bibitem{estate} T. Ceccherini-Silberstein, F. Scarabotti and F.
Tolli, Finite Gelfand pairs and their applications to probability
and statistics. (Russian) \textit{Sovrem. Mat. Prilozh.} No. 27,
Algebra i Geom. (2005), 95--140; translation in \textit{J. Math.
Sci.} (N.Y.), \textbf{141} (2007), no. 2, 1182--1229.

\bibitem{crestednoi} D. D'Angeli and A. Donno, Crested products of Markov chains, \textit{Ann. Appl. Probab.} \textbf{19} (2009), no. 1,
414--453.

\bibitem{noiortogonal} D. D'Angeli and A. Donno, Markov chains on orthogonal block
structures, in \textit{European J. Combin.} {\bf 31} (2010), no.
1, 34--46.

\bibitem{diaconis} P. Diaconis, {\it Group Representations in Probability and
Statistics}, Institute of Mathematical Statistics Lecture
Notes--Monograph Series, 11. Institute of Mathematical Statistics,
Hayward, CA, 1988.

\bibitem{saloffexc} P. Diaconis and L. Saloff-Coste, Comparison
theorems for reversible Markov chains, \textit{Ann. Appl.
Probab.}, {\bf 3} (1993), no. 3, 696--730.

\bibitem{insetto} A. Figà-Talamanca, An application of Gelfand
pairs to a problem of diffusion in compact ultrametric spaces, in:
Topics in Probability and Lie Groups: Boundary Theory, 51--67,
\textit{CRM Proc. Lecture Notes}, 28, Amer. Math. Soc.,
Providence, RI, 2001.

\bibitem{liggett} T. M. Liggett, {\it Interacting particle systems}, Grundlehren der Mathematischen
Wissenschaften [Fundamental Principles of Mathematical Sciences],
276. Springer--Verlag, New York, 1985.

\end{thebibliography}
\end{document}